\documentclass[12pt]{article}

\usepackage{amsmath}
\usepackage{amsthm}
\usepackage{amssymb}

\theoremstyle{plain}
\newtheorem{thm}{\protect\theoremname}
  \theoremstyle{plain}
  \newtheorem{lem}[thm]{\protect\lemmaname}
  \theoremstyle{plain}
  
  \theoremstyle{plain}

 \theoremstyle{plain}
\newtheorem{dfn}[thm]{\protect\definitionname}

\makeatother

  \providecommand{\lemmaname}{Lemma}
  \providecommand{\propositionname}{Proposition}
\providecommand{\theoremname}{Theorem}
\providecommand{\corollaryname}{Corollary}
\providecommand{\definitionname}{Definition}

\def\ol{\overline}

\newcommand{\reg}{\mathrm{reg}}
\newcommand{\sur}{\mathrm{sur}}

\def\R{\mathbb{R}}

\def\eps{\varepsilon}
\def\la{\lambda}
\def\grf{\mathrm{Gr}\,}

\def\Ri{\mathbb{R}\cup\{+\infty\}}

\begin{document}

\title{On Characterizations of Metric Regularity\\ of Multi-valued Maps
\footnote{Research  supported by the Scientific Fund of Sofia University under grant 80-10-133/25.04.2018.}}
%% use optional labels to link authors explicitly to addresses:
%% \author[label1,label2]{}
%% \address[label1]{}
%% \address[label2]{}

\author{M. Ivanov and N. Zlateva}
\date{\today}

\maketitle

\bigskip

\bigskip

\centerline{\textsl{Dedicated to Professor Alexander D. Ioffe}}

\bigskip

\bigskip

\begin{abstract}
 We provide a new proof along the lines of the recent book of A. Ioffe of a  1990's result of H. Frankowska showing that metric regularity of a multi-valued map can be characterized by regularity of its  contingent variation -- a  notion extending contingent derivative.
\end{abstract}



%% \linenumbers

%% main text
\section{Introduction}
Metric regularity, as well as,  the equivalent to it linear openness and pseudo-Lipschitz property of the inverse, are very important concepts in Variational Analysis. They have been intensively studied as it can be seen in a number of recent monographs, e.g. \cite{borwein-zhu-book, Penot-book, DR_book, Mordukhovich-book} and the references therein. A very rich and instructive survey on metric regularity is the book of A. Ioffe~\cite{Ioffe_book}.

 It may be noted in Chapter V of \cite{Ioffe_book} that the modulus of regularity of a multi-valued map between Banach spaces is estimated in terms of the tangential cones to its graph. The estimates are precise, but they are not characteristic. This is because in infinite dimensions a map may well be regular and the tangential cones to its graph be insufficiently informative, for details see \cite{franqui}.

 In \cite{Frankowska} H. Frankowska   introduced  the notion of contingent variation of a multi-valued map which extends  Bouligand tangential cone. This  notion can precisely characterize metric regularity.

 Let $(X,d)$ and $(Y,d)$  be  metric spaces and let
 $$
 F:X\rightrightarrows Y
 $$
 be a multi-valued map. If $V\subset Y$ the restriction $F^V$ is defined by
 $$
 	F^V(x) := F(x)\cap V,\quad\forall x\in X,
 $$
 see \cite[p.54]{Ioffe_book}. The properties related to the so restricted map are called \textit{restricted}.

 For example, the multi-valued map $F:X\rightrightarrows Y$ is called \textit{restrictedly Milyutin regular} on $(U,V)$, where $U\subset X$ and $V\subset Y$, if there exists a number $r>0$ such that
 	$$
 		B(v,rt)\cap V\subset F(B(x,t))
 	$$
 	whenever $(x,v)\in \grf F\cap (U\times V)$ and $B(x,t)\subset U$, where $B(x,t)$ is the closed ball with center $x$ and radius $t$: $B(x,t):=\{ u\in X:\ d(u,x)\le t \}$, and $\grf F = \{ (x,v):\ v\in F(x) \}$.
 	
 	The supremum of all such $r$ is called \textit{modulus of surjection},  denoted by
 	$$
 		\sur_mF^V(U|V).
 	$$
 	By convention, $\sur_mF^V(U|V) = 0$ means that $F$ is not restrictedly Milyutin regular on $(U,V)$.

This notion taken from \cite{Ioffe_book} is explained in great detail in Section~\ref{sec:milutin} below.

\medskip
In the literature, e.g. \cite[Section~5.2]{Ioffe_book}, there are various estimates of $\sur_mF^V(U|V)$ and related moduli in terms of  derivative-like objects. Unlike the so called \textit{co-derivative criterion}, see \cite[Section~5.2.3]{Ioffe_book}, most of the \textit{primal} estimates are not characteristic in general.  Here we re-establish one primal criterion which complements \cite[Section~5.2.2]{Ioffe_book} and is, moreover, characteristic. It is essentially done by H. Frankowska in \cite{Frankowska}, see also \cite{franqui}. There a new derivative-like object is defined as follows.

\smallskip
	Let $(X,d)$ be a metric space, $(Y, \|\cdot \|)$ be a Banach  space,
	$
	F:X\rightrightarrows Y
	$
	be a multi-valued map.
 For $(x,y)\in \grf F$ the  \textit{contingent variation} of $F$ at $(x,y)$ is the closed set
  	$$
	F ^{(1)} (x,y):= \limsup _{t\to 0^+} \frac{F(B(x,t))-y}{t},
	$$
where $\limsup$ stands for the Kuratowski limit superior of sets.

Equivalently, $v\in F ^{(1)} (x,y)$ exactly when there exist a sequence of reals $t_n\downarrow 0$ and a sequence $(x_n,y_n)\in \grf F$ such that
$d(x,x_n)\le t_n$ and
\[
\left \| v-\frac {y_n-y}{t_n}\right \| \to 0, \mbox { when } n\to \infty.
\]

%In \cite{NMEarxiv} we work with the following

%\begin{dfn}
%  Let $(X,d)$ and $(Y,d)$  be linear metric spaces,
%  $$
%    F:X\rightrightarrows Y
%  $$
%  be a multi-valued map, and $A\subset X$ be a non-empty set. For $(x,y)\in \grf F:=\{(x,y):\ y\in F(x)\}$ define $D'F (x;y)(A)\subset Y$ by
%  $$
%    v\in D'F (x,y)(A)\iff \exists t_n\to 0^+:\ \lim_{n\to \infty} \frac{\inf d (F(x+t_nA),y+t_nv)}{t_n} = 0.
%  $$
%\end{dfn}

This notion extends the so-called contingent, or graphical, derivative usually denoted by $DF(x,y)$, e.g.  \cite[pp.163, 202]{Ioffe_book}.
%In the annotation of \cite{frankovska} and that if $(X,\|\cdot\|)$, $(Y, \|\cdot\|)$ are normed spaces, it holds that $F^{(1)}(x,y) = D'F(x,y)(B_X)$.

\medskip
Our main result can now be stated. As usual, $B_Y$ denotes the closed unit ball of the Banach space $(Y,\| \cdot \|)$.

\begin{thm}\label{F_result}
	Let $(X,d)$ be a metric space and $(Y,\| \cdot \|)$ be a Banach space, let $U\subset X$ and $V\subset Y$ be non-empty open sets. Let
	$$
	F:X\rightrightarrows Y
	$$
	be a multi-valued map with complete graph.
	
	$F$ is restrictedly Milytin regular on $(U,V)$ with $\sur _m F^V(U\vert V)\ge r>0$ if and only if
	\begin{equation}\label{eq:fr-1}
		F^{(1)}(x,v)\supset r B_Y\ \mbox{ for all }(x,v)\in \grf F \cap (U\times V).
	\end{equation}	
\end{thm}

This result is essentially established by H. Frankowska in \cite[Theorem 6.1 and Corollary 6.2]{Frankowska}. However, there it is presented as a characterization of local modulus of regularity in terms of the local variant of the condition~\eqref{eq:fr-1}. Here we render the characterization global. The technique in \cite{Frankowska} is different, but it again depends on Ekeland Variational Principle.

\smallskip

The rest of the article is organized as follows. In Section~\ref{sec:milutin} we provide for reader's convenience the relevant material from \cite{Ioffe_book}. We also present in another form the first  criterion for Milyutin regularity from \cite{Ioffe_book}. In Section~\ref{sec:main} we prove Theorem~\ref{F_result}.

\section{Milyutin regularity}\label{sec:milutin}

Let $(X,d)$ and $(Y,d)$ be metric spaces.
Let $U\subset X$ and $V\subset Y$, let $F:X \rightrightarrows Y$ be a multi-valued map and let $\gamma (\cdot)$ be extended real-valued function on $X$ assuming positive values (possibly infinite) on $U$.

\begin{dfn}(\textbf{linear openness}, \cite[Definition 2.21]{Ioffe_book})
$F$ is said to be $\gamma$-open at linear rate on $(U,V)$ if there is an $r>0$ such that
\[
B(F(x),rt)\cap V\subset F(B(x,t)),
\]
if $x\in U$ and $t<\gamma (x)$, i.e.
\[
B(v,rt)\cap V\subset F(B(x,t)),
\]
whenever $(x,v)\in \grf F$, $x\in U$ and $t<\gamma (x)$.
\end{dfn}

Denote by $\sur _\gamma F(U|V)$ the upper bound of all such $r>0$ and call it  \emph{modulus of $\gamma$-surjection} of $F$ on $(U,V)$. If no such $r$ exists, set $\sur _\gamma F(U|V){=}0$.

\begin{dfn}(\textbf{metric regularity}, \cite[Definition 2.22]{Ioffe_book})
$F$ is said to be $\gamma$-metrically regular on $(U,V)$ if there is  $\kappa >0$ such that
\[
d(x,F^{-1}(y))\le \kappa d(y,F(x)),
\]
provided  $x\in U$, $y\in V$ and $\kappa d(y,F(x))<\gamma (x)$.
\end{dfn}

Denote by $\reg _\gamma F(U|V)$ the lower bound of all such $\kappa >0$ and call it  \emph{modulus of $\gamma$-metric regularity} of $F$ on $(U,V)$. If no such $\kappa$ exists, set $\reg _\gamma F(U|V)=\infty $.

\begin{thm}(\textbf{equivalence theorem}, \cite[Theorem 2.25]{Ioffe_book})\label{Theorem 2.25 Ioffe}
The following are equivalent for any metric spaces $X$, $Y$, any $F:X\rightrightarrows Y$, any $U\subset X$, $V\subset Y$ and any extended real-valued function $\gamma (\cdot)$ which is positive on $U$:

a) $F$ is $\gamma$-open at linear rate un $(U,V)$;

b) $F$ is $\gamma$-metrically regular on $(U,V)$.

Moreover (under the convention $0.\infty =1$),
\[
\sur _\gamma F(U|V).\reg _\gamma F(U|V)=1.
\]
\end{thm}

\begin{dfn}(\textbf{regularity}, \cite[Definition 2.26]{Ioffe_book})
We say that $F:X\rightrightarrows Y$ is  $\gamma$-regular on $(U,V)$ if the equivalent properties of Theorem~\ref{Theorem 2.25 Ioffe} are satisfied.
\end{dfn}

\begin{dfn}(\textbf{Miluytin regularity}, \cite[Definition 2.28]{Ioffe_book})
Set
\[
m_U(x):=d(x, X\setminus U).
\]
We shall say that $F$ is Milyutin regular on $(U, V)$ if it is $\gamma$-regular on $(U,V)$ with $\gamma (x)=m_U(x)$.
\end{dfn}

\smallskip

We will need also \textbf{Ekeland Variational Principle} (see \cite[p.45]{Phelps}): Let $(M,d)$ be a complete metric space, and $f:M\to \Ri$ be a proper, lower semicontinuous and bounded from below function. Assume that $f(\ol x)\le \inf f+\la \eps$ for some $\ol x\in M$ and $\la \eps >0$. Then there is
$\ol y\in M$ such that

(i) $f(\ol y)\le f(\ol x)-\la d(\ol x,\ol y)$;

(ii) $d(\ol x,\ol y)\le \eps$;

(iii) $f(x)+\la d(x,\ol y)\ge f(\ol y)$, for all $x\in M$.

\medskip

The following characterization of Milyutin regularity is very similar in form (in fact equivalent) to the so called \textbf{first criterion for Milyutin regularity}, see \cite[Theorem 2.47]{Ioffe_book}. It is also similar to \cite[Proposition~2.2]{cibul},  but there it is stated in local form. We present here a proof for reader's convenience.

Following \cite[p.35]{Ioffe_book} for $\xi>0$ we denote by $d_\xi$ the product metric
\begin{equation}\label{def:d:xi}
	d_\xi ((x_1,y_1),(x_2,y_2)):=\max \{ d(x_1,x_2),\xi d(y_1,y_2)\},
\end{equation}
where  $x_i\in X$, $y_i\in Y$, $i=1,2$, and $(X,d)$ and $(Y,d)$ are metric spaces.

\begin{thm}\label{new}
	Let $(X,d)$, $(Y,d)$ be metric spaces.  Let $F:X\rightrightarrows Y$ be a multi-valued map with complete graph. Let $U\subset X$ and $V\subset Y$.
	Then
\[ \sur _mF(U|V) =  \sup \{ r\ge 0: \exists\ \xi>0\mbox{ such that  } \]
\[\forall (x,v)\in \grf F,\ x\in U,\ y\in V\mbox{ satisfying }
0<d(y,v)<rm_U(x)
\]
\begin{equation}\label{star}
\exists (u,w)\in \grf F \mbox{ such that } d(y,w)<d(y,v)-rd_\xi((x,v),(u,w))\}.
\end{equation}
\end{thm}

\begin{proof}
	Let us denote by $s_1$ the left hand side of the above equation, i.e. $s_1:=\sur _mF(U|V)$. In other words,
\[s_1=\sup \{ r\ge 0: B(v,rt)\cap V\subset F(B(x,t)),\ \forall (x,v)\in \grf F,\ x\in U, t<m_U (x)\}.\]
	
	Denote by $s_2$ the right hand side of the equation.
	
	We need to show that $s_1 = s_2$.
	
\medskip
	
	First, we will show that $s_1\le s_2$.
	
	If $s_1=0$ we have nothing to prove.
	
	Let $s_1>0$. Take $0<r<r'<s_1$. Let $x\in U$, $v\in F(x)$ be fixed. Let $y\in V$ be such that $0<d(y,v)<rm_U(x)$. In particular $0<d(y,v)<r'm_U(x)$. Set $\displaystyle t:=\frac{d(y,v)}{r'}$. Then $t<m_U(x)$. By $r'<s_1= \sur_mF(U|V)$ and by the definition of $\sur_mF(U|V)$ it holds that $y\in B(v,r't)\cap V\subset F(B(x,t))$, i.e. $y\in F(B(x,t))$. So, there exists $u\in B(x,t)$ such that $y\in F(u)$.
	
	Fix $\xi$ such that $0<\xi r'<1$. Then
	\[
	d_\xi((x,v),(u,y))=\max \{d(x,u),\xi d(v,y)\}\le\max \{t,\xi r't\}=t\max \{1,\xi r'\}=t,
	\]
	so
	\[
	r'd_\xi((x,v),(u,y))\le r't=d(y,v).
	\]
	Observe that $d_\xi((x,v),(u,y))>0$ since $d(v,y)>0$. The latter and $r'>r$ yield
	\[
	rd_\xi((x,v),(u,y))< r't<d(y,v),
	\]
	or
	\[
	0<d(y,v)-rd_\xi((x,v),(u,y)).
	\]
	Since $0=d(y,y)$ we get that
	\[
	d(y,y)<d(y,v)-rd_\xi((x,v),(u,y))
	\]
	and \eqref{star} holds with $w=y$ as $(u,y)\in \grf F$.
	
	This means that $r\le s_2$. Finally, $s_1\le s_2$.
	
\medskip

	Second, we will prove that $s_2\le s_1$.

If $s_2=0$ we have nothing to prove.
	
	Let now $s_2>0$. Let $0<r<s_2$. Let us fix $x_0\in U$, $v_0\in F(x_0)$ and $0<t<m_U(x_0)$.
	
	Fix $y\in V$ such that $d(y,v_0)\le rt$, i.e. $y\in B(v_0,rt)\cap V$. Let $M:= \grf F$, and let $\xi>0$ correspond to $r$ in the definition of $s_2$. It is clear that $(M,d_\xi)$ is a complete metric space.
	
	Consider the function $f:M\to \R$ defined as $f(u,w):=d(w,y)$.
	
	Then $f\ge 0$ and it is continuous on $M$. Since $f(x_0,v_0)=d(v_0,y)\le rt$, by Ekeland Variational Principle there exists $(x_1,v_1)\in M$ such that
	
	(i) $f(x_1,v_1)\le f(x_0,v_0)-r d_\xi ((x_1,v_1),(x_0,v_0))$;
	
	(ii) $d_\xi ((x_1,v_1),(x_0,v_0))\le t$;

	(iii) $f(u,w)+rd_\xi ((u,w),(x_1,v_1))\ge f(x_1,v_1)$, for all $(u,w)\in M$.
	
	Or, equivalently
	
	(i) $d(v_1,y)\le d(v_0,y)-r d_\xi ((x_1,v_1),(x_0,v_0))\le rt-r d_\xi ((x_1,v_1),(x_0,v_0))$;
	
	(ii) $d(x_1,x_0)\le t,\quad \xi d(v_1,v_0)\le t$;
	
	(iii) $d(w,y)+rd_\xi ((u,w),(x_1,v_1))\ge d(v_1,y)$, for all $(u,w)\in M$.
	
	Set $p:=d(v_1,y)$.

	Assume that $p>0$. Take $t'$ such that $t<t'<m_U(x_0)$. For  $\displaystyle x\in B\left (x_1,\frac{p}{r}+t'-t\right)$ we have that 	
	\begin{eqnarray*}
		d(x,x_0)&\le &d(x,x_1)+d(x_1,x_0)\\
		&\le &\frac{p}{r}+t'-t+d(x_1,x_0)\\
	\mbox{(using (i))}	&\le &\frac{rt-rd(x_1,x_0)}{r}+t'-t+d(x_1,x_0)\\
		&=&t-d(x_1,x_0)+t'-t+ d(x_1,x_0)\\
		&=&t'.
	\end{eqnarray*}
	Hence $\displaystyle B\left(x_1,\frac{p}{r}+t'-t\right)\subset B(x_0,t')\subset U$. Then $\displaystyle \frac{p}{r}+t'-t\le m_U(x_1)$, and  $\displaystyle \frac{p}{r}< m_U(x_1)$ because  $t'-t>0$. Hence, $0< d(v_1,y)<rm_U(x_1)$. But now \eqref{star} contradicts (iii).

Therefore, $p=0$ and then $y=v_1\in F(x_1)$. Since by (ii) $x_1\in B(x_0,t)$, we have $y\in F(B(x_0,t))\cap V$.
	
 Since $x_0\in U$, $v_0\in F(x_0)$, $y\in B(v_0,rt)\cap V$ and $0<t<m_U(x_0)$ were arbitrary, this means that $r\le s_1$. Since $0<r<s_2$ was arbitrary, $s_2\le s_1$, and the proof is completed.
\end{proof}

\medskip

In the definitions of regularity properties it is not required that $F(x)\subset V$. Such requirements can be included in the definitions as follows.

\begin{dfn}(\textbf{restricted regularity}, \cite[Definition 2.35]{Ioffe_book})
Set $F^V(x):=F(x)\cap V$. We define restricted $\gamma$-openness at linear rate and restricted $\gamma$-metric regularity on $(U,V)$ by replacing $F$ by $F^V$.
\end{dfn}

The equivalence  Theorem~\ref{Theorem 2.25 Ioffe} also holds for the restricted versions of the properties. The case is the same with Theorem~\ref{new}, where the proof needs only small adjustments when working with $F^V$ instead of $F$.

\section{Proof of the main result}\label{sec:main}
 The proof of our main result relies on the following Lemma.

\begin{lem}\label{lem_for_criterion}
	Let $(X,d)$ be a metric space and $(Y,\| \cdot \|)$ be a Banach space, let $U\subset X$ and $V\subset Y$ be non-empty sets  and let
	$$
	F:X\rightrightarrows Y
	$$
	be a multi-valued map.
	
	If for some  $r>0$ it holds that
	\[
	 F^{(1)}(x,v)\supset r B_Y\ \mbox{ for all }(x,v)\in \grf F \cap (U\times V),
	\]
	then for  any $\displaystyle 0< r' <r$ and any $\xi \in (r^{-1}, (r')^{-1})$ it holds that for any $x\in U$ and any $v\in F^V(x)$ and $y\in V\setminus\{v\}$ %satisfying $0<\| y-v\|<r'm(x)$
	there is $(u,w)\in \grf F$ such that
	\[
	\| y- w\|<\| y-v\|  -r'd_\xi((x,v),(u,w)).
	\]
	\end{lem}

\begin{proof}
	Let $\displaystyle r' \in (0,r)$ be fixed.
	
	Fix $\xi >0$ such that $\displaystyle (r')^{-1}> \xi > r^{-1}$.

	Take $(x,v)$ such that $(x,v)\in \grf F\cap (U\times V)$.
	
	Fix $ y\in V$ such that $0<\| y-v\|  $.

	Set $\displaystyle \bar v:= r\frac{y-v}{\| y-v\|}$. Obviously $\|  \bar v \| =r$. By assumption,
	$F^{(1)}(x,v)\ni \bar v$. By definition of the contingent variation there exist $t_n\downarrow 0$,
	$u_n\in X$ as well as $w_n\in Y$ and $z_n \in Y$ such that $w_n\in F(u_n) $, $d(x,u_n)\le t_n$, $\| z_n\| \to 0$  and
	\begin{equation}\label{eq:u}
		v+t_n\bar v=w_n+ t_nz_n.
	\end{equation}
	
	Note first that for $n$ large enough
	\begin{equation}\label{eq:xy}
		\xi\|w_n-v\| > t_n \ge d(x,u_n)
		\Rightarrow
		d_\xi((x,v),(u_n,w_n))=\xi \|w_n-v\|.
	\end{equation}
	Indeed, $\|w_n-v\| = t_n\|\bar v - z_n\| \ge t_n (r-\|z_n\|)$ and, since $\xi(r-\|z_n\|)\to\xi r > 1$ as $n\to\infty$, we have $\xi\|w_n-v\| > t_n$ for $n$ large enough.
	
From \eqref{eq:u} we have
\begin{equation}\label{nnn}
y-w_n=y-v-t_n\ol v+t_nz_n.
\end{equation}

Since
$$
y-v-t_n\ol v = \left(1-t_nr\| y-v\| ^{-1}\right)(y-v),
$$
and since $1-t_nr\| y-v\| ^{-1}>0$ for $n$ large enough, we have for such $n$ that
$$
\|y-v-t_n\ol v \|= \left(1-t_nr\| y-v\| ^{-1}\right)\| y-v\| = \| y-v\| - t_nr.
$$
Combining the latter with \eqref{nnn} we get for $n$ large enough
\begin{eqnarray}\label{trn}
		\|y-w_n\|&=&\|y-v-t_n\bar v +t_nz_n\| \nonumber\\
                 &\le &\|y-v-t_n\bar v\| +t_n\|z_n\| \nonumber\\
                 &=&\| y-v\| -t_n(r-\| z_n\|).
\end{eqnarray}

	On the other hand, \eqref{eq:u} can be rewritten as $w_n -v = t_n\bar v - t_nz_n$, hence
 $$
 \| w_n-v\| =t_n\| \ol v-z_n\| \le t_n (r+\| z_n\| ),
 $$
and using this estimate we obtain that
	$$
		\liminf_{n\to\infty} \frac{t_n(r-\|z_n\|)}{r'\xi\|v-w_n\|} \ge \liminf_{n\to\infty} \frac{t_n(r-\|z_n\|)}{r'\xi t_n(r+\| z_n\|)} = \frac{1}{r'\xi} > 1.
	$$
From this and \eqref{trn} we have that for large $n$
	$$
		\|y-w_n\| < \|y-v\| - r'\xi\|v-w_n\| .
	$$
Using \eqref{eq:xy} we finally obtain that for all $n$ large enough
	$$
		\|y-w_n\| < \|y-v\| -r'd_\xi((x,v),(u_n,w_n))
	$$
	and the claim follows.
\end{proof}

Proving our main result is now straightforward.

\setcounter{thm}{0}

\begin{thm}
Let $(X,d)$ be a metric space and $(Y,\| \cdot \|)$ be a Banach space, let $U\subset X$ and $V\subset Y$ be non-empty open sets. Let
	$$
	F:X\rightrightarrows Y
	$$
	be a multi-valued map with complete graph.
	
	$F$ is restrictedly Milytin regular on $(U,V)$ with $\sur _m F^V(U\vert V)\ge r>0$ if and only if
	\[
	 F^{(1)}(x,v)\supset r B_Y\ \mbox{ for all }(x,v)\in \grf F \cap (U\times V).
	\]
\end{thm}

\begin{proof}
Let
\[
	F^{(1)}(x,v)\supset r B_Y\ \mbox{ for all }(x,v)\in \grf F \cap (U\times V).
	\]
From  Lemma~\ref{lem_for_criterion} and Theorem~\ref{new} it follows that $\sur _m F^V(U\vert V)\ge r$.

Conversely, let $\sur _m F^V(U\vert V)\ge r>0$. This means that
\[
B(v,rt)\cap V \subset F(B(x,t))
\]
whenever $(x,v)\in \grf F$, $x\in U$, $v\in V$ and $t<m_U(x)$.

Take arbitrary $(x,v)\in \grf F^V$, $x\in U$ and note that $m_U(x)>0$ because $U$ is open. Take  positive $t$ such that  $t<m_U(x)$.

For any $y\in rB_Y $ it holds that $v+ty\in B(v,rt)$. Moreover,  $v+ty\in V$ will be true for small $t$ because $V$ is open. Then, by assumption, $v+ty \in F(B(x,t))$, so
$\displaystyle y\in \frac{F(B(x,t))-v}{t}$ which means that $y\in F^{(1)}(x,v)$. Hence, $F^{(1)}(x,v)\supset rB_Y$.
\end{proof}

\bigskip

\noindent\textbf{Acknowledgements.} We wish to express our gratitude for the interesting discussions we have had with Professor Ioffe in the summer of 2018 on some topics in his recent monograph \cite{Ioffe_book}, and for his kind attention.

\end{document}